\documentclass{article}

\usepackage{amsmath,amsthm,amsfonts,amssymb,mathabx,latexsym,enumerate,graphicx,mathrsfs}
\usepackage[utf8]{inputenc}

\usepackage{hyperref}
\hypersetup{colorlinks, citecolor=blue, linkcolor=black}

\usepackage{todonotes}

\usepackage{tikz}
\usetikzlibrary{arrows,decorations.pathmorphing,decorations.pathreplacing,patterns,shapes.symbols}
\newcommand{\wvert}[1]{\filldraw[draw=black,fill=white] (#1) circle (0.1);}
\newcommand{\bvert}[1]{\filldraw (#1) circle (0.1);}
\newcommand{\sbvert}[1]{\filldraw (#1) circle (0.05);}
\newcommand{\chunk}[1]{
  \begin{scope}[xshift=#1]
    \draw[fill=gray!20] (-2,0) -- (2,0) -- (2,-2) -- (-2,-2)--cycle (0,0) -- (0,-2);
    \node [draw, dashed, cloud, cloud puffs=11, minimum width = 2cm, minimum height=1cm, fill=gray!60,opacity=0.9]
    at (0,0) {};
  \end{scope}
}
\newcommand{\chunkII}[1]{
  \begin{scope}[xshift=#1]
    \draw[fill=gray!20,pattern=horizontal lines] (-2,0) -- (2,0) -- (2,-2) -- (-2,-2)--cycle (0,0) -- (0,-2);
    \node [draw, dashed, cloud, cloud puffs=11, minimum width = 2cm, minimum height=1cm, fill=gray!60,opacity=0.9]
    at (0,0) {};
  \end{scope}
}

\newtheorem{thm}{Theorem}[section]
\newtheorem{cor}[thm]{Corollary}
 
\newtheorem{lem}[thm]{Lemma}

\theoremstyle{definition}
\newtheorem{defn}[thm]{Definition}

\theoremstyle{remark}

\newcommand{\define}{\textit}

\newcommand {\calC} {{\mathcal {C}}}

\newcommand {\calF} {{\mathcal {F}}}   
   
\newcommand {\calH} {{\mathcal {H}}}

\newcommand {\calL} {{\mathcal {L}}}

\newcommand {\calS} {{\mathcal {S}}}

\newcommand {\calV} {{\mathcal {V}}}

\newcommand {\calf} {{\mathcal {F}}}

\newcommand{\mo}{{-1}}
\newcommand{\onto}{\ensuremath{\twoheadrightarrow}}
\newcommand{\bk}[1]{\langle #1 \rangle}
\newcommand{\into}{\ensuremath{\hookrightarrow}}

\newcommand{\actson}{\curvearrowright}

\newcommand{\cay}[2]{{\mathrm{Cay}\left(#1,#2\right)}}
\newcommand{\edges}[1]{{\ensuremath{\mathrm{Edges}\left(#1\right)}}}
\newcommand{\verts}[1]{{\ensuremath{\mathrm{Vertices}\left(#1\right)}}}

\newcommand{\cleave}[1]{\ensuremath{\widecheck{#1}}}
\newcommand{\hcarrier}[2]{\ensuremath{{\mathcal{H}_{#1}\left(#2\right)}}}
\newcommand{\dsfac}{\preccurlyeq}
\newcommand{\pdsfac}{\prec}
\newcommand{\cafsd}{\succ}

\title{On the one-endedness of graphs of groups.}
\author{Nicholas Touikan}

\begin{document}

\maketitle
\begin{abstract}
  We give a technical result that implies a straightforward necessary
  and sufficient conditions for a graph of groups with virtually
  cyclic edge groups to be one ended. For arbitrary graphs of groups,
  we show that if their fundamental group is not one-ended, then we
  can blow up vertex groups to graphs of groups with simpler vertex
  and edge groups. As an application, we generalize a theorem of
  Swarup to decompositions of virtually free groups.
\end{abstract}

\section{Introduction}

A finitely generated group $G =\bk{S}$ is said to be
\define{one-ended} if the corresponding Cayley graph $\cay G S$ cannot
be separated into two or more infinite components by removing a finite
subset. Otherwise $G$ is said to be \define{many ended}. It is a
classical result due to Stallings \cite{Stallings-gt3dm} that a
many-ended group either decomposes as an amalgamated free product, or
an HNN extension, over a finite group.

Given the Bass-Serre correspondence between group actions on
simplicial trees and their decompositions, or splittings, as
(fundamental groups of) graphs of groups, c.f. \cite{Serre-arbres}, a
finitely generated group $G$ is many ended if and only if it acts
minimally, without inversions, and cocompactly on a simplicial tree
$T$ in which for some edge $e$ the stabilizer $G_e$ is finite.

It is often the case that a graph of groups with many ended vertex
groups is itself one ended. For example, the fundamental group of a
closed surface is one ended but it is an amalgamated free product of
free groups, which are many ended.  Theorem \ref{thm:main}, stated and
proved in Section \ref{sec:main}, essentially characterizes one ended
graphs of groups. This result is rather technical, but has many
``non-technical'' corollaries that we will now present.

We say that $G$ is \define{one ended relative to a collection $\calH$
  of subgroups} if for any minimal non-trivial $G$-tree $T$ with
finite edge stabilizers, there exists a subgroup $H \in \calH$ that
acts without a global fixed point. Or, equivalently, $G$ is
\define{many ended relative to $\calH$} if $G$ admits a non-trivial
splitting as a graph of groups relative to $\calH$ (i.e. groups in
$\calH$ are conjugate into vertex groups) with finite edge groups.

\begin{cor}\label{cor:double-over-cyclic}
  If $G_1$ is one ended relative to a collection $\calH_1 \cup
  \{C_1\}$, and $G_2$ is one ended relative to the $\calH_2 \cup
  \{C_2\}$ with $C_1 \approx C_2$ virtually cyclic groups, then any
  free product with amalgamation of the form\[
  G_1*_{C_1=C_2}G_2
\] is one ended relative to $\calH_1 \cup \calH_2$.
\end{cor}

In the case of graphs of free groups with cyclic edge groups, this
corollary (actually its natural generalization, c.f. Corollary
\ref{cor:graph-of-two-ended}) is proved in \cite[Theorem
18]{Wilton-one-ended} and implied by results in
\cite{Diao-Feighn-free-split}. Corollary \ref{cor:double-over-cyclic}
is false if we do not require the amalgamating subgroups to be
virtually cyclic or, synonymously, two ended. Nonetheless, we can still
understand the failure of one endedness of general graphs of groups.

\begin{defn}
  A $G$-equivariant map $S \to T$ of simplicial $G$-trees is called a
  \define{collapse}, if $T$ is obtained by identifying some edge
  orbits of $S$ to points. In this case we also say that $S$ is
  obtained from $T$ by a \define{blow up}. We call the preimage
  $\cleave T_v \subset S$ of a vertex $v\in T$ its \define{blowup}.
\end{defn}

\begin{defn}\label{defn:accessible}
  We write $H \dsfac G$ to signify that $G$ splits essentially
    as a graph of groups with finite edge groups and $H$ is a vertex
    group. A group $G$ is \define{accessible}, if it admits no
    infinite proper chains\[
    G \cafsd G_1\cafsd G_2 \cafsd \ldots
    \]
\end{defn}

For example, if $F$ is a free group and $H\dsfac F$, then $H$ is a
free factor of $F$. This next theorem, a formal consequence of Theorem
\ref{thm:main}, states that if a graph of groups with finitely
generated infinite edge group is not one ended, then we can blow up
some of its vertex groups.

\begin{thm}\label{thm:blowup}
  If $T$ is a $G$-tree (in which a collection of subgroups $\calH$ act
  elliptically) with infinite edge groups and $G$ is not one-ended
  (relative to $\calH$) then there is a vertex $v \in \verts T$ and an
  edge $e \in \edges T$ with $v \in e$ such that the orbit of $v$ can
  be blown up with $G_v$ acting minimally on the non-trivial blow ups
  $\cleave T_v$ satisfying the following properties:
  \begin{itemize}
  \item $G_e \leq G_v$ is the stabilizer of a vertex in $\cleave
    T_v$.
  \item The edge groups of $\cleave T_v$ are conjugate in $G_v$ to the
    vertex groups of an essential amalgamated free product or HNN
    decomposition of $G_e$ with a finite edge group.
  \end{itemize}
  In particular, in the tree $S$ obtained by blowing up the orbit of
  $v$ in $T$ to $\cleave T_v$, each vertex or edge stabilizer of $S$
  is $\dsfac$ a vertex or edge stabilizer of $T$, and at least one of
  these inclusions is a strict. Furthermore the groups in $\calH$ act
  elliptically on $S$.
\end{thm}

We note that blowing up a $G$-tree is equivalent to \emph{refining} a
graphs of groups. If $G$ acts on a tree with accessible vertex and
edge stabilizers then the order $\pdsfac$ actually tells us that the
vertex groups of the blowup given by Theorem \ref{thm:blowup} have
lower complexity, in the sense that the process of successively
blowing up vertex groups in this manner must terminate in finitely
many steps.

Accessible groups, in turn, are abundant: Linnell in
\cite{Linnell-acc} showed that if there is a global bound on the order
of finite order elements in a finitely generated group, then the group
is accessible. Dunwoody in \cite{Dunwoody-1985} showed that finitely
presented groups are accessible. We now use Theorem \ref{thm:blowup}
to give a proof of Corollary \ref{cor:double-over-cyclic}:

\begin{proof}[Proof of Corollary \ref{cor:double-over-cyclic}]
  We show the contrapositive. Let $T$ be the Bass-Serre tree dual to
  the splitting $G=G_1*_C G_2$, and suppose that $G$ is not one-ended
  relative to $\calH=\calH_1 \cup \calH_2$.  Note that any
  decomposition of a virtually cyclic group as an HNN extension or an
  essential amalgamated free product must have finite edge groups. It
  follows that in all cases, by Theorem \ref{thm:blowup}, some orbit
  of vertices $Gv$ can be blown up to minimal $gG_vg^\mo$-trees with
  finite edge groups, implying that one of the vertex groups $G_i$
  fixing some vertex $v \in \verts{T}$ acts minimally on $\cleave T_v$
  with finite edge stabilizers with \[\calH_i = \left\{ H \in \calH
    \mid H \cap G_i \neq \{1\}\right\}\] and $C_i = G_e$ for some $v
  \in e \in \edges{T}$ acting elliptically. It follows that $G_i$ is
  not one ended relative to $\calH_i \cup \{C\}$.
\end{proof}

This proof is easily adapted to give:

\begin{cor}\label{cor:graph-of-two-ended}
  The fundamental group $G$ of a graph groups with two-ended edge
  groups is one ended (relative to a collection $\calH$ of subgroups) if
  and only if every vertex group $G_v$ is one ended relative to the
  incident edge groups (and the collection $\{H^g \cap G_v \mid g\in
  G, H \in \calH\}$.)
\end{cor}

Using the full strength of Theorem \ref{thm:main}, we will also
generalize a result of Swarup on the decomposition of free groups
\cite{Swarup-1986} to virtually free groups. This result was already
partially generalized by Cashen \cite{Cashen-virt-free} to
decompositions of virtually free groups with virtually cyclic edge
groups.

\begin{thm}\label{thm:virt-free}
  Let $G$ be finitely generated and virtually free.
  \begin{enumerate}
  \item\label{it:VF-A} If $G$ splits as an amalgamated free product
    $G=A*_CB$ with $C$ finitely generated and infinite then there is
    some $C_1 \dsfac C$ such that $C_1 \dsfac A$ or $C_1 \dsfac B$.
    
  \item\label{it:VF-H} If $G$ splits as an HNN extension $G=A*_{C,t}$
    with $C$ finitely generated and infinite, then there is an
    infinite subgroup $C_1\dsfac C$ and a splitting $\Delta$ of $A$ as
    a graph of groups with finite edge groups relative to $\{C_1,t^\mo
    C_1t\}$ such that either $C_1$ or $t^\mo C_1t$ is a vertex group
    of $\Delta$.
  \end{enumerate}
\end{thm}

Unlike in Swarup's proof, we do not use homological methods. Our proof
is more along the lines of the geometric arguments found in
\cite{Wilton-one-ended, Louder-III, Bestvina-Feighn-outer-limits,
  Diao-Feighn-free-split} using graphs of spaces $X$ with $\pi_1(X) =
G$. The presence of torsion, however, can make the attaching maps in
the graphs of spaces difficult to describe. By using the more abstract
$G$-cocompact core of the product of two $G$-trees
\cite{Guirardel-coeurs}, we sidestep these difficulties. The core has
been used before to study pairs of group splittings.In particular,
Fujiwara and Papasoglu in \cite{F-P-JSJ} use it to show the existence
of QH subgroups for one ended groups that have hyperbolic-hyperbolic
pairs of slender splittings; this is the main technicality in
constructing group theoretical JSJ decompositions. Although it could
be noted that the action of our group on the core gives rise to a
$G$-orbihedron à la \cite{Hae}, we will not need this machinery; in
fact, modulo classical Bass-Serre theory and Guirardel's Core Theorem
for simplicial trees, Theorem \ref{thm:core} (of which we sketch a
proof), our argument is self-contained.

\subsection{Acknowledgements}

I wish to thank John MacKay and Alessandro Sisto for asking me for a
proof of Corollary \ref{cor:double-over-cyclic}. I had actually even
given them what I thought to be a counterexample; their
countercounterexample gave me ample motivation to investigate this
problem further. This paper also would not have been possible without
everything I learned from Lars Louder. The ideas of Section
\ref{sec:preliminaries}, especially the usefulness of the Core
Theorem, arose from our discussions while working on strong
accessibility. I am also grateful for the meticulous work of the
anonymous referee who spotted many tiny mistakes, as well as a couple
embarrassing ones, and gave suggestions that substantially improved
the exposition.  Finally, I thank Inna Bumagin. This paper was written
while I was supported as a postdoctoral fellow by her NSERC grant.

\section{Preliminaries}\label{sec:preliminaries}

\subsection{Group actions}

All group actions will be from the left. Let $X$ be a $G$-set. If
$S\subset X$ is a subset, we will denote by $G_S$ the (setwise)
stabilizer $\{g \in G \mid gS = S\}$. If $S = \{x\}$ is a singleton,
then we will write $G_x$ instead of $G_{\{x\}}$. We call a subset
$S\subset X$ \define{$G$-regular} if for any $x,y \in S$ in the same
$G$-orbit there is some $g \in G_S$ such that $gx=y$. The following
lemma is immediate:

\begin{lem}\label{lem:regular}
  Let $X$ be a $G$-set. If $S\subset X$ is $G$-regular, then we have
  an embedding\[ G_S\backslash S \into G\backslash X.
\]
\end{lem}

In this paper, all trees will be simplicial. In particular we will
consider them to be topological spaces, equipped with a CW-structure,
which also makes them into graphs. We further metrize these graphs by
viewing edges as real intervals of length 1.  We say a $G$-tree $T$ is
\define{without inversions} if, for any edge $e \in \edges T$, if
$ge=e$ then $g$ fixes $e$ pointwise. Equivalently, if $u,v \in \verts
T$ are the vertices at the ends of the edge $e$, then we have
inclusions \[G_u \geq G_e \leq G_v.\] We call vertex and edge
stabilizers, vertex groups and edge groups respectively. All $G$-trees
will be without inversions. We assume the reader is familiar with
Bass-Serre theory and we will switch freely between $G$-trees and
splittings as graphs of groups, viewing the two as being equivalent.

Let $T$ be a $G$-tree. $T$ is \define{essential} if every edge of $T$
divides it into two infinite components. We say a $G$-tree $T$ is
\define{without inversions} if, for any edge $e \in \edges T$, if
$ge=e$ then $g$ fixes $e$ pointwise.  We say that $T$ is
\define{minimal} if there are no proper subtrees $S\subset T$ with
$G_S=G$. We say $T$ is \define{cocompact} if $G\backslash T$ is
compact. An element $g$ or a subgroup $H$ of $G$ are said to
\define{act elliptically on $T$} if the groups $\bk g$ or $H$ fix some
$v \in \verts T$.

\subsection{Products of trees, cores, and leaf spaces}
If $T_1$ and $T_2$ are $G$-trees, then we have a natural
induced action $G\actson T_1 \times T_2$. Since the trees
$T_1,T_2$ are 1 dimensional CW complexes, we may consider their product
$T_1 \times T_2$ as a \define{square complex}, i.e. a 2 dimensional CW
complex whose cells consist of vertices, edges, and squares. There are
natural projections $p_i:T_1 \times T_2 \onto T_i$.  The following
Lemma is immediate:

\begin{lem}\label{lem:prod-no-inv}
  If the actions $G \actson T_1$ and $G\actson T_2$ are without
  inversions, then so is the action $G \actson T_1 \times T_2$,
  i.e. if $\sigma \supset \epsilon$ is an inclusion of cells
  (e.g. a square containing an edge), then $G_\sigma \leq
  G_\epsilon$.

  If the collection of subgroups $\calH$ act elliptically on $T_1$ and
  $T_2$ then each subgroup of $\calH$ fixes a vertex of $T_1\times
  T_2$.
\end{lem}

The action $G \actson T_1\times T_2$ is not cocompact in general. It
turns out, however, that we can extract a useful subset, namely
Guirardel's cocompact core. We state the special case of his result
applied to simplicial trees.

\begin{thm}[The Core Theorem {\cite[c.f. Théorème principal and
    Corollaire 8.2]{Guirardel-coeurs}}]\label{thm:core}
  
  Let ${G\actson T_1}$, ${G\actson T_2}$ be two minimal actions of a
  finitely generated group $G$ on simplicial trees $T_1, T_2$ with
  finitely generated edge stabilizers. Suppose furthermore that
  $T_1,T_2$ do not equivariantly collapse to a common non-trivial tree.

  Then there is a $G$-invariant subset $\calC \subset T_1\times T_2$
  called the \define{core of the action} $G\actson T_1\times T_2$
  which is defined as the smallest connected $G$-invariant subset such
  that the restrictions of the projections $p_i|\calC:\calC \onto T_i$
  have connected fibres. The quotient $\calS = G\backslash \calC$ is
  compact.
\end{thm}

Suppose for the rest of this section that $T_1,T_2$ satisfy the
hypotheses of Theorem \ref{thm:core}.  The restrictions of the
projections $p_i|_\sigma: \sigma \to T_i$ are well defined for each
cell (i.e. a vertex, edge, square) $\sigma \subset T_1\times T_2$. If
$\sigma$ is a square then the projection is onto an edge $p_i(\sigma)
\in \edges T_i$. If $\lambda_1,\lambda_2 \subset \sigma$ are two
fibers of such a projection, see Figure \ref{fig:fibers}
\begin{figure}[h]
  \centering
  \begin{tikzpicture}
  \draw[pattern=horizontal lines] (0,0) -- (0,1) -- (1,1) -- (1,0)
  --cycle;
  \draw[->] (1.5,0.5) -- node[above]{$p_i$} (2.5,0.5);
  \draw (3,0) -- (3,1);
  \draw[very thick] (0,1) -- (1,1) (0,0) -- (1,0);
  \sbvert{3,0}
  \sbvert{3,1}
\end{tikzpicture}
  \caption{The projection of a square on an edge and some of its fibers}
  \label{fig:fibers}
\end{figure}
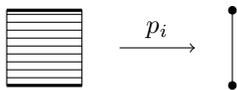, we can define a distance
$d^\sigma_i(\lambda_1,\lambda_2)$ to be the distance in $p_i(\sigma)$
between the points $p_i(\lambda_1)$ $p_i(\lambda_2)$; thus putting a
metric $d^\sigma_i$ on the set of $p_i$-fibers in a cell $\sigma$. We
now define the \define{$i$-leaf space $\calL_i$} of a subset $Z
\subset T_1\times T_2$ to be the set of connected unions of $p_i$-fibers
of cells in $Z$, called \define{leaves}, so that we see $Z$ as being
\emph{foliated} by the leaves in $\calL_i$. $\calL_i$ is a 1-complex
with metrized edges; therefore we can endow $\calL_i$ with the path
metric $d_i$. As a consequence of the direct product structure we have
the following.

\begin{lem}\label{lem:leaf-spaces-forests}
  If $Z \subset T_1\times T_2$, then the leaf spaces $\calL_i$ are
  forests (see Figure \ref{fig:leaf-space}).
\end{lem}

\begin{figure}[h]\centering
  \begin{tikzpicture}[scale=0.5]
  \draw[very thick, pattern=horizontal lines] (0,0) -- (0,2) -- (2,2) -- (2,0) --cycle;
  \draw[very thick,pattern=horizontal lines] (2,0) -- (2,2) -- (4,2) -- (4,0) --cycle;
  \draw[very thick, pattern=horizontal lines] (2,2) -- (4,2) -- (6,3)
  -- (4,3) -- cycle;
  \draw[fill=white](2,2) -- (1,3) -- (3,3) -- (4,2) --cycle;
  \draw[very thick,pattern=horizontal lines] (2,2) -- (1,3) -- (3,3)
  -- (4,2) --cycle;
  \draw[very thick] (7,3) -- (8,2) -- (8,0) (8,2) --(9,3);
  \bvert{7,3}
  \bvert{8,2}
  \bvert{8,0}
  \bvert{9,3}
\end{tikzpicture}
\caption{The $i$-leaves in a square complex and the resulting leaf
  space, which is a tree.}
\label{fig:leaf-space}
\end{figure}
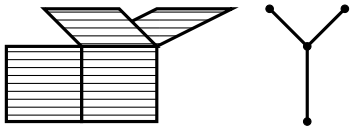

If $\calC \subset T_1\times T_2$ is a core then the leaf spaces
$\calL_i$ are homeomorphic to the trees $T_i$. Later, however, we will
be performing operations that will alter the leaf spaces.

\subsection{Induced splittings}

Let $v \in \verts{T_i}, e\in \edges{T_i}$ and let $m_e$ be the
midpoint of $e$. Let $\tau_v = p_i^{\mo}(\{v\})\cap \calC$ and $\tau_e
= p_i^{\mo}(\{m_e\})\cap \calC$. By Theorem \ref{thm:core} the preimages
$\tau_v,\tau_e$ are connected and are therefore leaves in $\calL_i$.

Since we have an action $G \actson \calC$, since $\tau_v,\tau_e$ are
defined as $T_i$-point preimages via a $G$-equivariant map, and since
$G_v,G_e$ are exactly the stabilizers of these points $v, m_e$, the
subsets $\tau_v,\tau_e \leq \calC$ are $G$-regular so by Lemma
\ref{lem:regular} we have embeddings\[ G_v\backslash \tau_v \into
G\backslash \calC \hookleftarrow G_e\backslash \tau_e.
\] By Theorem \ref{thm:core}, $G\backslash \calC$ is compact so
the quotients $G_v\backslash \tau_v, G_v\backslash \tau_v$ must be as
well. Moreover, because $\tau_v,\tau_e$ are contained in $p_i$-fibres,
for $j \neq i$ the restrictions\[ p_j|\tau_v:\tau_v \to T_j,
p_j|\tau_e:\tau_e \to T_j
\] are injective. Finally, the projection $p_j|_\calC:\calC \onto T_j$ is
$G$-equivariant; we have shown the following.

\begin{lem}\label{lem:induced-splittings}
  If $v \in \verts{T_i}, e\in \edges{T_i}, j\neq i$, then the fibres
  $\tau_v, \tau_e$ are mapped injectively via $p_j$ to subtrees that
  are $G_v,G_e$-invariant (respectively). Viewed as subsets of the
  core $\calC \subset T_1\times T_2$, $\tau_v$ and $\tau_e$ coincide
  with their $j$-leaf spaces.

  The actions $G_e\actson \tau_e, G_v\actson \tau_v$ are cocompact.
  Moreover $\tau_v,\tau_e$ are infinite if and only if the actions of
  the subgroups $G_v\actson T_j,G_e\actson T_j$ are without global
  fixed points.
\end{lem}

The $G_v,G_e$-trees $\tau_v,\tau_e$ give splittings induced by the
action on $T_j$. The blowups of Theorem \ref{thm:main} will be
obtained by modifying the trees $\tau_v$. For aficionados of CAT(0)
cube complexes, it is worth remarking that the core $\calC$ is a
CAT(0) square complex, in fact a $\calV\calH$-complex, and that the set
of fibres $\tau_e, e \in \edges{T_i}$ is the set of hyperplanes.

\subsection{Spurs, free faces, and cleavings}

In the previous section we obtained cocompact $G_v,G_e$-trees
$\tau_v,\tau_e$. We say a tree has a \define{spur} if it has a vertex
of degree 1. An edge adjacent to a spur is called a \define{hair}. We
now give a shaving process.

\begin{lem}\label{lem:spur-removal}
  Let $T$ be a cocompact $G$-tree. $T$ is minimal if $T$ doesn't have
  any spurs. If $T$ is not minimal, then we can obtain the minimal
  subtree $T(G)$ as the final term of a finite sequence\[ T=T_0,
  \ldots, T_k = T(G),
  \] where $T_{i+1}$ is obtained from $T_i$ by $G$-equivariantly
  contracting one $G$-orbit of hairs to points.
\end{lem}
\begin{proof}
  Let $v \in \verts T$ be a spur adjacent to an edge $e \in \edges T$
  and let $u \in \verts T$ be the other endpoint of $e$. The map $T\to
  T$ obtained by $G$-equivariantly collapsing $ge$ onto $gu; g \in G$
  is a deformation retraction onto a proper $G$ invariant subtree, so
  $T$ is not minimal.

  Suppose now that $T$ is not minimal. Then there is some proper
  $G$-invariant subtree $S \subset T$. Let $K$ be the closure of some
  connected component of $T \setminus S$. Then $K \cap S = \{v\}$ for
  some $v \in \verts S$. Since $S$ is $G$-invariant and connected, we
  must have $G_K \leq G_v$. It follows that for any $w \in \verts K$
  and any $g \in G_K$ the distance $d_T(w,v) = d_T(gw,v)$, i.e. the
  action of $G_K$ on $K$ is the action on a rooted tree with root
  $v$. Since $K$ is $G$-regular, we have an embedding $G_K\backslash
  K\into G\backslash T$ which is compact; thus $K$ must have finite
  radius since $G_K$ preserves distances from the root.
  
  Since $K$ is a rooted tree with finite diameter it must have a
  non-root vertex of valence 1. By the argument at the beginning of
  the proof we can $G_K$-equivariantly collapse hairs and since $G_K
  \actson K$ is cocompact, after finitely many collapses we will have
  collapsed $K$ to $v$. Again since $G\actson T$ is cocompact, there
  are only finitely many orbits of connected components of $T\setminus
  S$, so the result follows.
\end{proof}

If $\sigma$ is a square in some $Z \subset T_1\times T_2$, then we say
an edge $\epsilon \subset \sigma$ is a \define{free face} if it only
contained in one square. The following terminology is due to Wise
\cite{Wise-cubulate-small-canc}.

\begin{defn}\label{defn:hypercarrier}
  Let $e \in \edges{T_i}$ and let $\tau_e \subset \calC$ be the fibre
  mentioned in Lemma \ref{lem:induced-splittings}. The
  \define{hypercarrier} $\hcarrier{\calC}{\tau_e}$ is the union of
  squares of $\calC$ intersecting $\tau_e$ non-trivially.
\end{defn}

We note that for $e \in T_i$, a hypercarrier is mapped to an edge of
$T_i$ and that $\hcarrier{\calC}{\tau_e}$ is homeomorphic to $\tau_e \times [-1,1]$.

\begin{defn}\label{defn:i-transverse}
  We say an edge $\epsilon$ in some $Z\subset T_1\times T_2$ is
  $i$-transverse if it coincides with its $i$-leaf space, or
  equivalently it is mapped monomorphically via $p_i|\epsilon$, or
  equivalently if it is contained in a $j$-leaf.
\end{defn}

An immediate consequence of Lemma \ref{lem:spur-removal} and Figure
\ref{fig:spur-free-face} is the following.
 
\begin{lem}\label{lem:free-edges}
  Let $e \in \edges{T_i}$, if $G_e\actson \tau_e$ is not minimal then
  $\hcarrier{\calC}{\tau_e}$ has a square $\sigma$ containing an
  $i$-transverse free face $\epsilon$.
\end{lem}

\begin{figure}[h]
  \centering
  \begin{tikzpicture}[scale=0.5]
    \draw[thick, pattern=horizontal lines] (0,0) -- (0,2) --
    (8,2) -- (8,0) --cycle;
    \draw[thick] (2,2) -- (2,0) (4,2) -- (4,0) (6,2) -- (6,0);
    \draw[very thick] (0,1) -- (8,1);    
    \draw[fill=white] (4,2) -- (5,1) -- (5,-1) -- (4,0) --cycle;
    \draw[thick, pattern=north west lines] (4,2) -- (5,1) -- (5,-1) --
    (4,0) --cycle;
    \draw[very thick] (4,1) -- (5,0);
    \draw[very thick] (5,1) -- (5,-1);
    \bvert{0,1} \bvert{2,1} \bvert{4,1} \bvert{6,1} \bvert{8,1}
    \bvert{5,0}
    \draw (5,-0.5) node[right]{$\epsilon$};
     \draw (0,1) node[left]{$\ldots $} (8,1) node[right]{$\ldots $};
     \draw (-3,2)  -- node[left]{$e$} (-3,0);
     \bvert{-3,2} 
     \bvert{-3,0}
     \draw[->] (-1,1) --node[above]{$p_i$} (-2,1);
  \end{tikzpicture}
  \caption{A spur of $\tau_e$ and the corresponding free face
    $\epsilon$ in the hypercarrier $\hcarrier{\calC}{\tau_e}$.}
\label{fig:spur-free-face}
\end{figure}
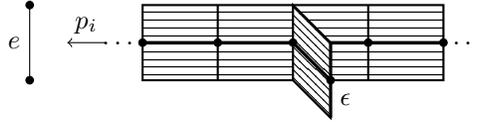

We now borrow some terminology from \cite{Diao-Feighn-free-split}.

\begin{defn}
  A simplicial map $S \to T$ between two trees that is obtained by
  identifying edges sharing a common vertex is called a
  \define{folding}. If $T$ is obtained from $S$ by a folding, then we
  say $S$ is obtained from $T$ by a \define{cleaving.}
\end{defn}

We now have the following which is immediate (see Figure
\ref{fig:collapse-n-cleave}). 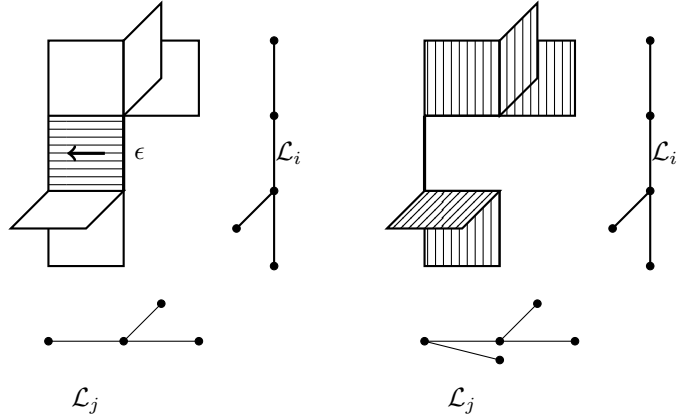
\begin{figure}[h]
  \centering
  \begin{tikzpicture}[scale=0.5]
    \draw[thick] (0,-2) -- (0,4) -- (2,4) -- (2,-2) --cycle;    
    \draw[thick] (2,4) -- (4,4) -- (4,2) -- (2,2) --cycle;
    \draw[fill=white] (2,4) -- (3,5) -- (3,3) -- (2,2) -- cycle; 
    \draw[thick] (2,4) -- (3,5) -- (3,3) -- (2,2) -- cycle; 
    \draw[fill=white] (0,0) -- (-1,-1) -- (1,-1) -- (2,0) --cycle;
    \draw[thick] (0,0) -- (-1,-1) -- (1,-1) -- (2,0) --cycle;
    \draw[thick] (0,2) -- (2,2);
    \draw[very thick] (2,0) -- node[right]{$\epsilon$} (2,2);
    \draw[pattern = horizontal lines] (0,0) -- (0,2) -- (2,2) --
    (2,0);
    \draw[very thick,->] (1.5,1) -- (0.5,1);
    \begin{scope}[xshift=6cm]
      \draw[thick] (0,-2) -- (0,4) (-1,-1) -- (0,0);
      \bvert{0,-2} \bvert{0,4} \bvert{-1,-1} \bvert{0,0} \bvert{0,2}
      \draw(1,1) node[left]{$\calL_i$};
    \end{scope}
    \begin{scope}[yshift=-4cm]
      \draw (0,0)--(4,0) (2,0)--(3,1);
      \bvert{0,0} \bvert{4,0}\bvert{2,0} \bvert{3,1}
      \draw (1,-1) node[below]{$\calL_j$};
    \end{scope}
    
    \begin{scope}[xshift = 10cm]
      \draw[thick,pattern=vertical lines] (0,-2) -- (0,0) -- (2,0) -- (2,-2) --cycle;    
      \draw[thick,pattern=vertical lines] (0,2) -- (0,4) -- (2,4) -- (2,2) --cycle;    
      \draw[thick,pattern=vertical lines] (2,4) -- (4,4) -- (4,2) -- (2,2) --cycle;
      \draw[fill=white] (2,4) -- (3,5) -- (3,3) -- (2,2) -- cycle; 
      \draw[thick,pattern=vertical lines] (2,4) -- (3,5) -- (3,3) -- (2,2) -- cycle; 
      \draw[fill=white] (0,0) -- (-1,-1) -- (1,-1) -- (2,0) --cycle;
      \draw[thick,pattern=north east lines] (0,0) -- (-1,-1) -- (1,-1) -- (2,0) --cycle;
      \draw[thick] (0,2) -- (2,2);
      \draw[very thick] (0,0) -- (0,2);
      \begin{scope}[xshift=6cm]
        \draw[thick] (0,-2) -- (0,4) (-1,-1) -- (0,0);
        \bvert{0,-2} \bvert{0,4} \bvert{-1,-1} \bvert{0,0} \bvert{0,2}
        \draw(1,1) node[left]{$\calL_i$};
      \end{scope}
      \begin{scope}[yshift=-4cm]
        \draw (0,0)--(4,0) (2,0)--(3,1);
        \draw (0,0) -- (2,-0.5);
        \bvert{0,0} \bvert{4,0}\bvert{2,0} \bvert{3,1} \bvert{2,-0.5}
        \draw (1,-1) node[below]{$\calL_j$};
      \end{scope}
    \end{scope}
    
  \end{tikzpicture}
  \caption{The effects of collapsing an $i$-transverse free face
    $\epsilon$: the leaf space $\calL_j$ gets cleaved, $\calL_i$
    remains unchanged. On the right the $j$-leaves are drawn.}
  \label{fig:collapse-n-cleave}
\end{figure}

\begin{lem}\label{lem:collapse-cleave}
  Let $\epsilon \subset Z \subset T_1\times T_2$ be an $i$-transverse
  free face in a square $\sigma$. If we collapse $\sigma$ onto the
  face opposite to $\epsilon$ the leaf space $\calL_i$ is unchanged
  and the leaf space $\calL_j$ gets cleaved.
\end{lem}

In fact this lemma can be used backwards to give a proof of Theorem
\ref{thm:core}. We will sketch it, leaving the details to an
interested reader familiar with folding sequences
\cite{BF-simplicial, Stallings-fold, Dunwoody-1998, KMW-2005}.

\begin{proof}[Sketch of the proof of Theorem
\ref{thm:core}]
Pick some vertex $v \in T_1\times T_2$ and consider its $G$-orbit. We
can add finitely many connected $G$-orbits of edges to get a connected
$G$-complex $Gv \subset \calC_1 \subset T_1\times T_2$. $\calC_1$ has
leaf spaces $\calL_1, \calL_2$ which project onto $T_1,T_2$. The
disconnectedness of the fibres of the projections $p_i|_{\calC_1}:
\calC_1 \onto T_i$ coincides with the failure of injectivity of the
projections $\calL_i \onto T_i$. By Lemma \ref{lem:collapse-cleave}
(backwards) adding a square can give a folding of one of the leaf
spaces. Since the edge groups of $T_1,T_2$ are finitely generated, and
because adding all the squares of $T_1\times T_2$ folds $\calL_i$ to
$T_i$, it follows that the leaf spaces $\calL_i$ can be made to
coincide with $T_i$ after adding finitely many $G$-orbits of squares.
\end{proof}

\section{The statement and proof of the main theorem}\label{sec:main}
For this section we fix a collection $\calH$ of subgroups of $G$. We
let $T_\infty,T_\calf$ be cocompact, minimal $G$-trees in which the
subgroups in $\calH$ act elliptically. We further require that edge
groups of $T_\infty$ are infinite and finitely generated and that edge
groups of $T_\calf$ are finite. Note that any non-trivial tree
obtained by a collapse of $T_\infty$ has infinite edge groups whereas
any collapse of $T_\calf$ has finite edge groups. It follows that
$T_\infty, T_\calf$, having no non-trivial common collapses, satisfy
the hypotheses of Theorem \ref{thm:core}.

\begin{thm}[Main Theorem]\label{thm:main}
  Let $\calH$ be a collection of subgroups of $G$ and let
  $T_\infty,T_\calf$ be cocompact, minimal $G$-trees in which the
  subgroups in $\calH$ act elliptically, suppose furthermore that the
  edge groups of $T_\calf$ are finite and that the edge groups of
  $T_\infty$ are infinite. Then there exists a vertex $v \in
  \verts{T_\infty}$ and a non-trivial, cocompact, minimal $G_v$-tree
  $\cleave T_v$ such that \begin{enumerate}[(i)]
  \item\label{it:star} for every $f \in \edges{T_\infty}$
    incident to $v$ the subgroups $G_f \leq G_v$ act elliptically on
    $\cleave{T}_v$, and 
  \item\label{it:still-elliptic} for every $H\in\calH, g\in G$ the subgroup
    $H^g\cap G_v\leq G_v$ acts elliptically on $\cleave T_v$.
  \end{enumerate} 
  Moreover,
  \begin{enumerate}[(1)]
  \item\label{it:finite-edge-gp-case} either every edge group of $\cleave T_v$ is finite; or
  \item\label{it:cleave-case} there is some edge $e\in
    \edges{T_\infty}$, incident to $v$, that not only satisfies
    (\ref{it:star}), but also satisfies the following:
  \begin{enumerate}
  \item\label{it:Ge-split} $G_e$ splits essentially as an amalgamated
    free product or an HNN extension with finite edge group.
  \item\label{it:Ge-stab} $G_e = G_{v_e}$ for some vertex $v_e \in \verts{\cleave{T}_v}$.
  \item\label{it:cleave-edge-groups} the edge stabilizers of $\cleave
    T_v$ are conjugate in $G_v$ to the vertex group(s) of the
    splitting of $G_e$ found in (\ref{it:Ge-split}); in particular the
    edge groups of $\cleave T_v$ are $\pdsfac$ $G_e$.
  \item\label{it:vertex-gps} The vertex groups of $\cleave T_v$ that
    are not conjugate in $G_v$ to $G_e$ are also vertex groups of a
    one edge splitting of $G_v$ with a finite edge group; in
    particular these vertex groups of $\cleave T_v$ are $\pdsfac G_v$.
  \end{enumerate}
\end{enumerate}
\end{thm}

An example of what happens in situation (\ref{it:cleave-case}) is
shown in Figure \ref{fig:operations-example}.

\begin{proof}
  Let $\calC$ be the core of $T_\infty\times T_\calf$. The
  $\infty$-leaf space $\calL_\infty$ is the tree $T_\infty$, and we
  can see $\calC$ as a tree of spaces (c.f. \cite{Scott-Wall} for
  details) which is a union of vertex spaces $\tau_v; v \in
  \verts{T_\infty}$ and edge spaces $\hcarrier{\calC}{\tau_e} = \tau_e \times
  [-1,1]; e \in \edges{T_\infty}$ attached to the $\tau_v$ along the
  subspaces $\tau_e\times \{\pm 1\}$.

  It may be that for some $e \in \edges{T_\infty}$ the $G_e$-trees
  $\tau_e$ are not minimal. By Lemmas \ref{lem:free-edges},
  \ref{lem:spur-removal}, \ref{lem:collapse-cleave}, we can repeatedly
  $G$-equivariantly collapse $\infty$-transverse free faces, so that
  after finitely many steps we obtain a \define{shaved core}
  $\calC'_s$ such that $\tau_e \cap \calC_s'$ are minimal $G_e$
  trees. Although the $\calf$-leaf space was cleaved repeatedly in the
  shaving process given by Lemma \ref{lem:spur-removal}, the
  $\infty$-leaf space is unchanged. We still write $\calL_\infty =
  T_\infty$.

  We will now construct a complex $\calC_s \subset \calC_s' \subset
  \calC$, called the \define{$\infty$-minimal core}. Its principal
  feature is that every tree $\tau_v\cap \calC_s$, $\tau_e\cap\calC_s$
  will be a minimal $G_v,G_e$-tree where $v \in \verts{T_\infty}$,
  $e\in \edges{T_\infty}$ respectively. Denote
  $\hcarrier{\calC'_S}{\tau_e} = \hcarrier{\calC}{\tau_e} \cap
  \calC'_s$. We call $\hcarrier{\calC'_S}{\tau_e}$ the
  \define{$\calC'_s$-hypercarrier} attached to a vertex space $\tau_v$
  in $\calC'_s$.  $\tau_e\cap\calC'_s$ naturally projects injectively
  into $\tau_v$ as a minimal $G_e$-invariant subtree where $G_e \leq
  G_v$. If $T$ is a $G$-tree and $H\leq G$, denoting by $T(S)$ the
  minimal $S$-invariant subtree, we have $T(H) \subset T(G)$. It
  therefore follows that all the $\calC'_s$-hypercarriers attached to
  $\tau_v$ are actually attached to the minimal $G_v$-invariant
  subtree of $\tau_v$. By Lemma \ref{lem:spur-removal}, after finitely
  many equivariant spur collapses we can make the vertex spaces
  $\tau_v$ into minimal $G_v$-trees. None of these collapses will
  affect the attached $\calC'_s$-hypercarriers
  $\hcarrier{\calC'_s}{\tau_e}$ and the leaf space $\calL_\infty =
  T_\infty$ is preserved. We have therefore constructed $\calC_s$, the
  $\infty$-minimal core. Denote $\hcarrier{\calC_s}{\tau_e} =
  \hcarrier{\calC'_s}{\tau_e} \cap \calC_s$. By what was written
  above, $\hcarrier{\calC_s}{\tau_e} = \hcarrier{\calC'_s}{\tau_e}$,
  and we now call $\hcarrier{\calC_s}{\tau_e}$ a
  \define{$\calC_s$-hypercarrier}.

  For every edge $k \in \edges{T_\calf}$, $G_k$ is finite, therefore a
  minimal $G_k$ tree is a point; thus, by cocompactness and
  regularity, the trees $\tau_k \in \calC$ have finite diameter and
  the same must be true of every connected component of $\tau_k \cap
  \calC_s$, so every connected component of $\tau_k \cap \calC_s$ has
  a spur. It therefore follows that $\calC_s$ must have an
  $\calf$-transverse free face $\epsilon$ containing a spur of some
  connected component of $\tau_k\cap \calC_s$ for some $k \in
  \edges{T_\calf}$. Furthermore the stabilizer $G_\epsilon \leq
  G_{p_\calf(\epsilon)}$ is an edge stabilizer of $T_\calf$; therefore
  it is finite.  This $\calF$-transverse free face $\epsilon$ must be
  contained in some $\tau_v\cap\calC_s; v\in
  \verts{T_\infty}$. Suppose first that $\epsilon$ was not contained
  in any $\calC_s$-hypercarrier attached to $\tau_v\cap\calC_s$. Then
  for every $e \ni v$ in $\edges{T_\infty}$, $G_e$ fixes some
  $\calC_s$-hypercarrier $\hcarrier{\calC_s}{\tau_e}$ such that
  $\hcarrier{\calC_s}{\tau_e} \cap \tau_v = \tau_e^+$ is contained in
  the complement $\left(\tau_v\cap\calC_s\right) \setminus
  G_v\epsilon$.

\begin{defn}\label{defn:non-e-tree}
  Let $T$ be a minimal $G$-tree and let $e \in \edges{T}$. We denote
  by $C(T,e)$, the \define{non-$e$-collapse of $T$}, the tree whose
  edges are the edges in the orbit $Ge\subset T$ and whose vertices
  are the closures of the connected components of $T \setminus Ge$,
  with $v \in \verts{C(T,e)}$ adjacent to $e \in \edges{C(T,e)}$ if
  and only if, viewed as subsets of $T$, $e\cap v \neq \emptyset$.
\end{defn}

It therefore follows that $\cleave T_v =
C\left(\tau_v\cap \calC_s,\epsilon\right)$ is a tree with finite edge groups, in which each
$G_e \leq G_v, e \in \edges{T_\infty}$ act elliptically, and also
conjugates of groups in $\calH$ intersecting $G_v$ act elliptically;
thus (\ref{it:star}), (\ref{it:still-elliptic}) and
(\ref{it:finite-edge-gp-case}) are satisfied.

Otherwise the free face $\epsilon \subset \tau_v\cap \calC_s$ is, by
definition of a free face, contained in \emph{exactly one}
$\calC_s$-hypercarrier $\hcarrier{\calC_s}{\tau_e}$. We will now
construct the $G_v$-tree $\cleave T_v$ satisfying
(\ref{it:cleave-case}). This construction is illustrated in Figure
\ref{fig:cleave-case}.
\begin{figure}[h]
  \centering
  \begin{tikzpicture}[scale=0.5]
  \begin{scope}[yshift=4cm]
    \chunk{0}
    \chunk{6cm}
    \chunk{-6cm}
    \draw[very thick] (-8,-2) -- (8,-2);
    \draw[pattern = vertical lines] (2,-2) --node[below]{$\overline{\epsilon}$} (4,-2) -- (4,0)
    --node[above]{$\epsilon$} (2,0) --cycle;
    \draw[very thick] (2,0) -- (4,0);
    \draw[very thick,->] (3,-0.5) -- (3,-1.5);
    \begin{scope}[xshift=-6cm]
      \draw[pattern = vertical lines] (2,-2) -- (4,-2) -- (4,0)
      -- (2,0) --cycle;
    \draw[very thick] (2,0) -- (4,0);
    \draw[very thick,->] (3,-0.5) -- (3,-1.5);
    \end{scope}
    \draw (8,0) node[right] {$\tau_v\cap \calC_s$};
    \draw (8,-2) node[right] {$\tau_e^-$};
    \draw (-8,-1) node[left] {$\ldots$} (8,-1) node[right] {$\ldots$};
 \end{scope}

  \chunkII{0}
  \draw (0,0) node{$C_0$};
  \chunkII{6cm}
  \draw (6,0) node{$C_1$};
  \chunkII{-6cm}
  \draw (-6,0) node{$C_{-1}$};
  \draw (-8,-2) -- (8,-2);
  \draw[line width = 0.1cm] (-8,-2) -- node[below]{$K_{-1}$} (-4,-2);
  \draw[line width = 0.1cm] (-2,-2) -- node[below]{$K_0$} (2,-2);
  \draw[line width = 0.1cm] (8,-2) -- node[below]{$K_1$} (4,-2);
  \draw (3,-2) node[below] {$\overline{\epsilon}$};
  \draw (-8,-1) node[left] {$\ldots$} (8,-1) node[right] {$\ldots$};
  \begin{scope}[yshift=-5cm]
    \draw[thick] (-6,0)node[above]{$v_{-1}$} -- (0,-2) --
    (0,0)node[above]{$v_{0}$} (6,0)node[above]{$v_0$} --
    (0,-2)node[below]{$v_e$};
    \wvert{-6,0} \wvert{6,0} \wvert{0,0} \bvert{0,-2}
    \draw (-8,-1) node[left] {$\ldots$} (8,-1) node[right] {$\ldots$};
  \end{scope}
\end{tikzpicture}

\caption{Constructing $\cleave T_v$. The top shows a portion of $Z$,
  the middle shows the result of equivariantly collapsing the free
  face $\epsilon$, the bottom shows the corresponding $\infty$-leaf
  space.}
  \label{fig:cleave-case}
\end{figure}
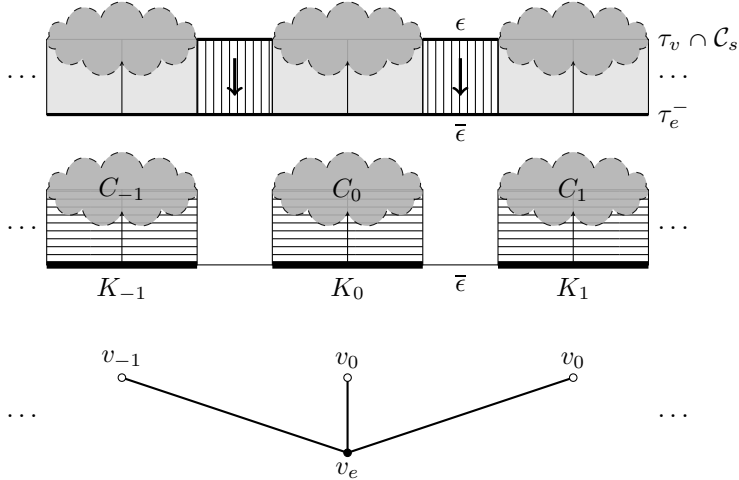
We first take the subset \[ Z = \left(\tau_v \bigcup_{e \ni v}
  \hcarrier{\calC_s}{\tau_e}\right)\bigcap \calC_s,\] i.e. $\tau_v\cap
\calC_s$ to which we attach all adjacent $\calC_s$-hypercarriers.  Now
the $G_v$-translates of $\epsilon$ are contained in the
$\calC_s$-hypercarriers $\hcarrier{\calC_s}{\tau_{ge}}; g \in
G_v$. For each such $\calC_s$-hypercarrier we denote by $\tau_{ge}^-$
the connected component of $\tau_e \times \{\pm 1\} \subset
\hcarrier{\calC_s}{\tau_{ge}}$ not contained in $\tau_v\cap \calC_s$ (see the top
of Figure \ref{fig:cleave-case}.)

We now $G_v$-equivariantly collapse the square $\sigma \supset
\epsilon$ onto the opposite side $\overline{\epsilon}$, to obtain a
connected $G_v$-subset $Z_c\subset Z$ (see the middle of Figure
\ref{fig:cleave-case}.) The resulting intersection $\tau_v \cap Z_c$
consists of a collection of connected components $\{C_i \mid i \in
I\}$. Similarly, the $G_e$-translates of ${\overline{\epsilon}}$ give
connected components $\{K_i\mid i \in I\}$ of $\tau_e \setminus
G_e\overline{\epsilon}$. Because $G_e$ acts on
$C(\tau_e^-,\overline{\epsilon})$, and by minimality of $\tau_e\cap \calC_s$, this
action is also minimal with one edge orbit. This gives us
(\ref{it:Ge-split}).

For every edge $v \in f \in \edges{T_\infty}$ that is not in the
$G_v$-orbit of $e$, the orbit $G_v\epsilon$ does not intersect
$\hcarrier{\calC_s}{\tau_f} \cap \tau_v$. It follows that each such $G_f \leq
G_v$ stabilizes some component $C_i$. We now detach from $Z_c$ all
$\calC_s$-hypercarriers not stabilized by a $G_v$-conjugate of $G_e$ to obtain a
$G_v$ complex $Z'_c \subset Z_c$, specifically\[ Z'_c = Z_c \bigcap
\left(\tau_v \bigcup_{g \in G_v} \hcarrier{\calC_s}{\tau_{ge}}\right)
\] Next we collapse each $G_v$-translate of $\tau_e^-$ to a vertex $v_e$,
collapse each component $C_i$ to a vertex $v_i$, and collapse each
connected component of $G_v$-translates of $\hcarrier{\tau_e}\cap
Z'_c$ onto an edge connecting $v_e$ and the corresponding vertex
$v_i$ to get the $G_v$-tree $\cleave T_v$. This is illustrated at the
bottom of Figure \ref{fig:cleave-case}.

Equivalently if we consider the leaf $\infty$-leaf space corresponding
to the union of the $\calC_s$-hypercarriers
$g\hcarrier{\calC_s}{\tau_e}; g \in G_v$ attached to
$\tau_v\cap\calC_s$, then we have a tree of radius 1, which is
$G_v$-isomorphic to $\{v\}\cup \left(\bigcup_{g\in G_v} ge\right)
\subset T_\infty$. After equivariantly collapsing the free face
$\epsilon$, Lemma \ref{lem:collapse-cleave} gives us a cleaving of
this radius 1 subtree to the infinite tree $\cleave T_v$ constructed
above. See Figure \ref{fig:cleave-leaf-space}.
\begin{figure}[htb]
  \centering
  \begin{tikzpicture}
    \draw (0,1) -- node[left]{$e$} (0,0) 
    (0,0) -- node[left]{$ge$}(0,-1);
    \bvert{0,1} 
    \bvert{0,-1} 
    \wvert{0,0}
    \draw (0,0) node[left]{$v$}; 
    \begin{scope}[xshift=7cm]
      \draw (-2,1)node[above]{$v_{e}$}--node[left]{$\cdots$}(-4,0) (-2,1)--(-2,0) (-2,1)--(0,0);
      \draw (-2,-1)node[below]{$v_{ge}$}--node[left]{$\cdots$}(-3,0)
      (-2,-1)--(-2,0) (-2,-1)--node[right]{$\cdots$} (-1,0);
      \draw (1,1)--node[left]{$\cdots$}(0,0) (1,1)--(1,0) (1,1)--node[right]{$\cdots$}(2,0);
      \wvert{-4,0}
      \wvert{-3,0}
      \wvert{-2,0}
      \wvert{-1,0}
      \wvert{0,0}
      \wvert{1,0}
      \wvert{2,0}
      \bvert{-2,1}
      \bvert{-2,-1}
      \bvert{1,1};
      \draw[ultra thick,->](-6.5,0.5) --node[above]{cleave} (-4.5,0.5);
      \draw[ultra thick,->](-4.5,-0.5) --node[below]{fold} (-6.5,-0.5);
    \end{scope}
    \end{tikzpicture}
    \caption{Equivariant collapsing free faces cleaves the leaf space
      of $Z'_C$ to a tree $\cleave T_v$ with infinite
      diameter.}
  \label{fig:cleave-leaf-space}
\end{figure}
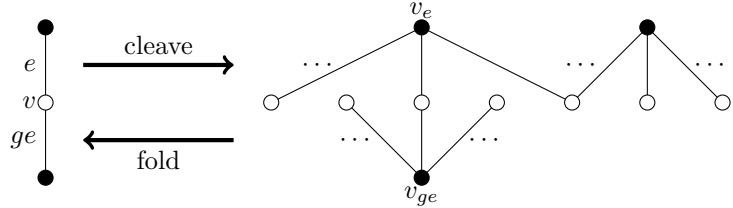
We note that if we took the $\infty$-leaf space of $Z_c$, i.e. had we
not detached the other hypercarriers, the resulting leaf space would
be a tree with many spurs. The tree $\cleave T_v$ we obtain is a minimal
$G_v$-tree that satisfies (\ref{it:Ge-stab}) and (\ref{it:star}).

We moreover note that, by construction, every
subgroup $H^g \cap G_v; g \in G, H \in \calH$ acts elliptically on
$\cleave T_v$; so (\ref{it:still-elliptic}) is satisfied as well.
 
The vertex stabilizers of $C(\tau_e^-,\overline{\epsilon})$
coincide with the component stabilizers $(G_e)_{K_i} = (G_v)_{K_i}$,
since $\tau_e^-$ is $G_v$-regular. We also have $(G_v)_{C_i} \cap
(G_v)_{\tau_e^-} = (G_v)_{K_i}$ (again see the middle of Figure
\ref{fig:cleave-case}.) It follows that the edges stabilizers of
$\cleave T_v$ satisfy (\ref{it:cleave-edge-groups}).

Finally note that the vertex groups of $\cleave T_v$ that are not
stabilized by $G_v$-conjugates of $G_e$ are also the vertex groups of
$C(\tau_v,\epsilon)$ (see the top of Figure \ref{fig:cleave-case}).
Finally, since $G_\epsilon$ is finite, (\ref{it:vertex-gps}) follows.
\end{proof}

\section{Splittings of virtually free groups}

Another way to use Theorem \ref{thm:main} is to obtain cleavings of
$G$-trees whose edge and vertex groups are ``smaller''. This will be
used as the inductive step in our proof of Theorem
\ref{thm:virt-free}.

\begin{figure}[h]
  \centering
  \begin{tikzpicture}[scale=0.5]
    \bvert{0,0} \bvert{2,0}
    \draw (0,0) node[above]{$A$} -- node[above]{$C$} (2,0)
    node[above]{$B$};

    \begin{scope}[xshift=8cm]
      \bvert{0,0} \bvert{2,0}
      \draw (0,0) node[above]{$A$} -- node[above]{$C$} (2,0)
      node[right]{$C$};
      \draw(2,0) -- node[above]{$C_1$} (4,1) node[right]{$B_1$}
      (2,0) -- node[below]{$C_2$} (4,-1) node[right]{$B_2$};
      \bvert{4,1} \bvert{4,-1}
      \draw[dashed] (3.5,0) circle (1.5);
      \draw[fill=gray, opacity=0.2] (3.5,0) circle (1.5);
    \end{scope}

    \begin{scope}[yshift=-6cm]
      \draw (0,0)node[above]{$A$} --node[above]{$C_1$} (2,1)
      node[above]{$B_1$}
      (0,0)node[above]{$A$} --node[below]{$C_2$} (2,-1)
      node[above]{$B_2$};
      \bvert{0,0} \bvert{2,1} \bvert{2,-1}
    \end{scope}

    \begin{scope}[xshift=10cm,yshift=-6cm]
      \bvert{0,0} \bvert{2,0}
    \draw (0,0) node[above]{$A$} -- node[above]{$C_1$} (2,0)
    node[above]{$B_1$};
    \end{scope}
    
    \draw[very thick,->] (3,0) --node[above]{Thm \ref{thm:main}}
    node[below]{Blow up $B$} (7,0);

    \draw[very thick, ->] (8,-1) -- node[above,sloped]{Cor
      \ref{cor:cleave-tree}} node[below,sloped]{Collapse $C$} (3,-4) ;
    
    \draw[very thick,dashed,->] (1,-1) -- node[above,sloped]{Cleave $C$}(1,-4);

    \draw[very thick, ->] (4,-6) --node[above]{Second construction}
    node[below]{Delete $C_2$} (8,-6);
  \end{tikzpicture}
  \caption{An example of the effects of Theorem \ref{thm:main},
    Corollary \ref{cor:cleave-tree}, and the second construction of
    the proof of Theorem \ref{thm:virt-free}, on a graph of
    groups. The vertices and edges are labeled by the corresponding
    vertex and edge groups. In all cases $B_i \pdsfac B$ and $C_i
    \pdsfac C$.}
  \label{fig:operations-example}
\end{figure}
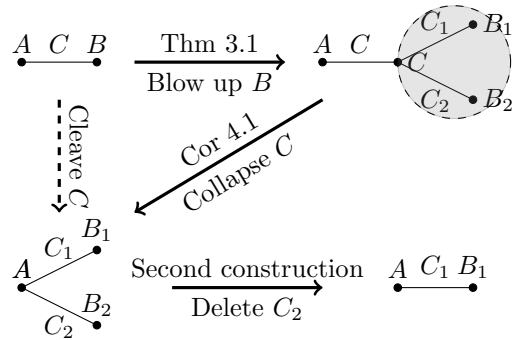
\begin{cor}\label{cor:cleave-tree}
  Let $T$ be a $G$-tree in which the subgroups $\calH$ act
  elliptically with infinite edge groups and let $G$ be many ended
  relative to $\calH$. Either some vertex $v \in \verts T$ can be
  blown up to a tree with finite edge groups; or there is an edge
  $e\in \edges{T}$ such that we can blow up $T$, relative to $\calH$,
  to some tree $\cleave T$, and then collapse the edges in the orbit
  of $e$ to points. The resulting tree $T'$ can also be obtained from
  $T$ by equivariantly cleaving some edge $e$. If $e'\in \edges{T'}$
  is a new edge obtained by a cleaving of $e$ then $G_{e'} \pdsfac
  G_e$. Also for each new vertex $v' \in \verts{T'}$ there is some $v
  \in \verts T$ that got cleaved such that $G_{v'} \pdsfac G_v$.

  Furthermore, in passing from $T$ to $T'$ the number of edge orbits
  and the number of vertex orbits does not decrease and increases by
  at most 1.
\end{cor}

\begin{proof}
  Suppose we are in case (\ref{it:cleave-case}) of Theorem
  \ref{thm:main}. Then some vertex $v$ gets blown up to $\cleave T_v$
  and some vertex stabilizer of $\cleave T_v$ coincides with
  $G_e$. Specifically $\cleave T$ can be obtained by deleting each
  blown up vertex $v$ from $T$ and then equivariantly reattaching
  every edge $e$ incident to $v$ to the corresponding vertex in
  $\cleave T_v$.

  In particular if $e \in \edges{T}$ is an the edge incident to $v$
  that satisfies (\ref{it:cleave-case}) of Theorem \ref{thm:main} then
  it is attached to the vertex $v_e \in \verts{\cleave T_v}$. We
  obtain $T'$ by collapsing the $G$-orbits of $e$ to points. This
  amounts to identifying the vertex $v_e$ to the vertex $u_e \in
  \verts{\cleave T}$ that is the other endpoint of $e$. From Figure 6
  it is clear that $T'$ is obtained by cleaving $T$.

  We finally note that in passing from $T$ to $\cleave T$ and then
  from $\cleave T$ to $T'$, the vertex and edge groups are
  non-increasing. Otherwise, the required properties of $T'$ are
  immediately satisfied by Theorem \ref{thm:main} (see Figure
  \ref{fig:operations-example}.)
\end{proof}

Finally, we can give our description of the decompositions of
virtually free groups as amalgamated free products or HNN extensions.

\begin{proof}[Proof of Theorem \ref{thm:virt-free}]
  We shall prove this result by successively applying Corollary
  \ref{cor:cleave-tree} until some desirable terminating condition is
  met. On one hand, virtually free groups have no one-ended subgroups
  so we will always be able to apply our Corollary; furthermore,
  virtually free groups are finitely presented. It now follows by Dunwoody
  accessibility \cite{Dunwoody-1985} that there are no infinite chains
  $C_1 \cafsd C_2 \cafsd \ldots$ of virtually free groups (recall
  Definition \ref{defn:accessible}) and that all such chains must
  terminate with finite groups. \\~
  
  \noindent \textbf{First construction} (pass to relatively one ended
  vertex subgroups): Let $T$ be a $G$ tree with one edge orbit $Ge$
  with $G_e$ infinite. By accessibility, we may pass to a tree
  $T^{(2)}$ obtained by blowing up some vertices $v$ of $T$ to trees
  $\cleave T_v$ such that the vertex groups of $\cleave T_v$ are
  either finite or one ended relative to the stabilizers $G_f$ of the
  incident edges $f \ni v$.  If possible, we take $T^{(1)}\subset
  T^{(2)}$ to be an infinite connected subtree obtained by deleting
  edges with finite stabilizers and we set $G^{(1)} = G_{T^{(1)}}$,
  i.e. the setwise stabilizer. We note that the vertex groups of
  $T^{(1)}$ are $\dsfac$ the vertex groups of $T$, and vertex groups
  are one ended relative to the incident edge groups. \\~
  
  \noindent \textbf{Second construction} (pass to smaller edge
  groups): The second construction utilizes Corollary
  \ref{cor:cleave-tree}. If $T_i$ is a $G_i$-tree with one edge orbit
  whose vertex groups are one ended relative to the incident edge
  groups, we first apply Theorem \ref{thm:main} to blow up a vertex $v
  \in \verts{T_i}$, and find ourselves in case (\ref{it:cleave-case})
  of the theorem. If $\cleave T_v$ has a finite edge group then $G_v$
  is not one-ended relative to the incident edge groups, contradicting
  our assumption.  By Corollary \ref{cor:cleave-tree} we can collapse
  an edge of the blowup of $T_i$ to get a cleaving $T_i'$ that has at
  most two edge orbits, with edge groups $\pdsfac$ the edge groups of
  $T_i$. The new vertex groups are also $\dsfac$ the old vertex
  groups. If there are two edge orbits we obtain $T_{i+1}\subset T_i'$
  as a maximal subtree containing only one edge orbit and we
  set $G_{i+1} = (G_i)_{T_{i+1}}$, i.e. we take the setwise
  stabilizer. See Figure \ref{fig:operations-example}. If $T'$ already
  has only one edge orbit then $T_{i+1} = T_i$ and $G_{i+1} = G_i$.
  \\~

  In both constructions, we pass to subgroups that split as graphs of
  groups such that the edge groups and vertex groups are $\dsfac$ the
  edge and vertex groups of the original splitting of the overgroup.

  We start with the amalgamated free product case. Let $T = T_0$ be
  the Bass-Serre tree corresponding to the splitting given in
  (\ref{it:VF-A}) of the statement of Theorem
  \ref{thm:virt-free}. Take the blow up $T_0^{(2)}$ obtained from the
  first construction. If one of the vertex groups of this blow-up
  coincides with an incident edge group we are done. Otherwise we may
  pass to the $G^{(1)}$ tree $T_0^{(1)}$ which still has one edge
  orbit, two vertex orbits and whose vertex groups are one ended
  relative to the incident edge groups. Furthermore because the new
  vertex groups are $\dsfac$ the vertex groups of $T$, if the
  statement of the theorem holds for $G^{(1)}$ and the splitting
  corresponding to its action on $T_0^{(1)}$ (which is also an
  amalgamated free product) then the statement also holds for $G$ and
  the splitting corresponding to its action on $T$.

  We can now apply our second construction to the $G_0^{(1)}$-tree
  $T_0^{(1)}$ to obtain a $G_1$-tree $T_1$, which again must have one
  edge orbit and two vertex orbits. Furthermore for the (conjugacy
  class) of the edge group, we have a proper containment $C_1\pdsfac
  C$. Again, because of the vertex groups of $T_1$ are $\dsfac$ the
  vertex groups of $T_0^{(1)}$ if the Theorem holds for this subgroup,
  it holds for $G$.

  We repeatedly apply our construction obtaining a sequence of groups
  that split as amalgamated free products. Each time we do the second
  construction we pass to a smaller edge group; so that by
  accessibility, at some point there is some subgroup $G_i$ acting on
  $T^{(2)}_i$ (see first construction) such that the vertex groups
  split as graphs of groups with finite edge groups and one of the
  incident edge groups coincides with the vertex group. So, since
  $\dsfac$ is transitive, (\ref{it:VF-A}) of Theorem
  \ref{thm:virt-free} is satisfied.

  We now tackle the HNN extension case. The proof goes the same way,
  we repeatedly blow up, cleave, and pass to subtrees; the main
  difference is that the $G$-tree $T$ has only one vertex orbit. If at
  some point one of the trees $T_i$ or $T_i^{(1)}$ has two vertex
  orbits, then these vertex groups are vertex groups of a splitting of
  the vertex group of $T_{i-1}$ with finite edge groups. It therefore
  follows that if $T_i$ satisfies (\ref{it:VF-A}) of Theorem
  \ref{thm:virt-free}, then $T_{i-1}$ satisfies (\ref{it:VF-H}) of
  Theorem \ref{thm:virt-free}; thus so must our original splitting
  $T$, by transitivity of $\dsfac$. Otherwise the proof goes through
  identically.
\end{proof}

\bibliographystyle{alpha}\bibliography{biblio.bib}
\end{document}